\documentclass{amsart}

\usepackage{amsrefs}
\usepackage{amssymb}
\usepackage{amsmath}
\usepackage{amsthm}
\usepackage{amsfonts}
\usepackage{epsfig}
\usepackage{graphicx}
\usepackage{psfrag}
\usepackage{color}
\usepackage{mathrsfs}
\usepackage[utf8]{inputenc}
\usepackage[T1]{fontenc}
\usepackage[english]{babel}
\usepackage{hyperref}

\hypersetup{
    colorlinks,
    citecolor=black,
    filecolor=black,
    linkcolor=black,
    urlcolor=black
}

\theoremstyle{plain}
\newtheorem{thm}{Theorem}[section]

\newtheorem{lem}[thm]{Lemma}
\newtheorem{prop}[thm]{Proposition}

\theoremstyle{definition}
\newtheorem{rmk}[thm]{Remark}

\newtheorem{df}[thm]{Definition}
\theoremstyle{remark}

\DeclareMathOperator{\id}{id}
\DeclareMathOperator{\ord}{ord}

\makeatletter
 \providecommand{\bigsqcap}{\mathop{\mathpalette\@updown\bigsqcup}}
 \newcommand*{\@updown}[2]{\rotatebox[origin=c]{180}{$\m@th#1#2$}}
\makeatother

\DeclareMathOperator{\tbigsqcap}{\textstyle\bigsqcap}
\DeclareMathOperator{\tbigsqcup}{\textstyle\bigsqcup}
\DeclareMathOperator{\tbigwedge}{\textstyle\bigwedge}

\newcommand{\s}{\mathcal S}

\renewcommand{\c}{\mathscr C}
\renewcommand{\l}{\mathscr L}
\renewcommand{\u}{\mathscr U}
\renewcommand{\v}{\mathscr V}
\newcommand{\w}{\mathscr W}
\renewcommand{\b}{\mathscr B}

\newcommand{\Z}{\mathbb Z}
\newcommand{\N}{\mathbb N}

\newcommand{\esp}{\vspace{.6em plus .15em minus .15em}}

\title{Expansive Systems on Lattices}

\author{M. Achigar}
\date{\today}

\begin{document}

\maketitle

\begin{abstract}
We study expansive dynamical systems in the setting of distributive lattices and their automorphisms, the usual notion of expansiveness for a homeomorphism of a compact metric space being the particular case when the lattice is the topology of the phase space ordered by inclusion and the automorphism the one induced by the homeomorphism, mapping open sets to open sets. We prove in this context generalizations of Mañé's Theorem and Utz's Theorem about the finite dimensionality of the phase space of an expansive system, and the finiteness of a space supporting a positively expansive homeomorphism, respectively. We also discuss the notion of entropy in this setting calculating it for the non-Hausdorff shifts.
\end{abstract}

\section{Introduction}\label{sec:intro}

Let $(X,d)$ be a metric space and $f\colon X\to X$ a homeomorphism. It is said that $f$ is \emph{expansive} iff there exist $\delta>0$, an \emph{expansivity constant}, such that if $x,y\in X$ and $d(f^nx,f^ny)<\delta$ for all $n\in\Z$ then $x=y$ (see \cite{Utz}*{\S1}). When $X$ is compact by \cite{Br62}*{Theorem 5} this is equivalent to \emph{uniform expansivity}, that is, for all $\varepsilon>0$ there exists $N\in\N$ such that if $x,y\in X$ and $d(f^nx,f^ny)<\delta$ for all $|n|\leq N$ then $d(x,y)<\varepsilon$. In that case, if we consider a finite open cover $\u$ of $X$ by sets of diameter less than $\delta$ we have that the following property holds:
\begin{equation}\label{ec:refexp}
\text{for any open cover $\v$ of $X$ there exists $N\in\N$ such that $\tbigwedge_{|n|\leq N}f^n\u\prec\v$,}
\end{equation}
where $\tbigwedge_{|n|\leq N}f^n\u$ denotes the collection of all subsets $U\subseteq X$ of the form $U=\bigcap_{|n|\leq N}f^nU_n$ with $U_{-N},\ldots,U_N\in\u$, and $\prec$ means that each of these subsets $U$ is contained in some member $V\in\v$. Indeed, given $\v$ let $\varepsilon>0$ a \emph{Lebesgue number} for $\v$, that is, if $U\subseteq X$ has diameter less or equal to $\varepsilon$ then $U$ is contained in some member of $\v$, and choose $N\in\N$ corresponding to this $\varepsilon$ in the definition of uniform expansivity above. Then, as the members of $\u$ has diameter less than $\delta$, the uniform expansivity property implies that the members of $\tbigwedge_{|n|\leq N}f^n\u$ has diameter less or equal to $\varepsilon$, hence $\tbigwedge_{|n|\leq N}f^n\u\prec\v$ as claimed. Conversely, is not difficult to see that the existence of an open cover $\u$ with the property (\ref{ec:refexp}) implies that $f$ is uniformly expansive, any Lebesgue number for $\u$ being an expansivity constant.

The above equivalent condition for the expansivity of a homeomorphism of a compact metric space in terms of open covers is interesting because it does not mention the metric, so that it has sense for any topological space non necessarily metrizable. This point of view is exploited in \cite{AAM} where the \emph{refinement expansive} homeomorphisms are introduced as those homeomorphisms of an arbitrary topological space admitting a finite open cover $\u$, called \emph{refinement expansivity cover}, having the property (\ref{ec:refexp}) above (see \cite{AAM}*{Definition 3.6 and Theorem 3.9}).

Other authors also considered topological definitions of expansivity not mentioning a metric, for example in \cite{KR} using \emph{generators} (see \cite{KR}*{Definition 2.4}), or in \cites{Fried,Br60} by means of closed neighborhoods isolating the diagonal in the square of the phase space (see \cite{Fried}*{p.\,489}, \cite{Br60}*{\S2 p.\,1163}). However all these definitions implies the metrizability of the phase space in the compact case (see \cite{KR}*{Corollary 2.8}, \cite{Fried}*{Lemma 2}, \cite{Br60}*{Corollary p.\,1164}), so that the concept of expansivity they manage is in fact equivalent to the usual one in the metric context. Refinement expansivity in turn does not imply metrizability and there are examples on non-Hausdorff spaces such as the \emph{non-Hausdorff shifts} introduced in \cite{AAM}*{\S 4.1}, whereas assuming the Hausdorff separation property refinement expansivity is equivalent to the usual metric expansivity.

In this work we exploit another important characteristic of the definition of refinement expansivity, that the other topological  definitions of expansivity discussed above does not have: it does not mention the points of the phase space, only the action of the homeomorphism on the open sets, that is, on the topology. Then refinement expansivity of a homeomorphism $f\colon X\to X$ of a topological space $(X,\tau)$ is actually a property of the map $\lambda_f\colon\l_X\to\l_X$, $\lambda_fU=f^{-1}U$, where $\l_X=\tau$. At this level we point out that in order to derive some properties of refinement expansive homeomorphisms we only need that $\l_X$ is closed by finite unions and intersections, and that $\lambda_f$ preserves them. Consequently we are moved to consider a definition of expansivity in the context of distributive lattices and their automorphisms, the topological context being a particular case.

In \S\ref{sec:exp} we introduce the expansive lattice automorphisms (Definition \ref{def:expnsive}) and prove some basic properties such as the invariance by conjugations and the expansivity of the powers (Propositions \ref{prop:exp_conjugado} and \ref{prop:exp_potencia}). In \S\ref{sec:mañé} Theorem \ref{teo:mañé} we generalize \emph{Mañé's Theorem} \cite{Ma} which asserts that if a compact metric space admits an expansive homeomorphism then it has finite topological dimension. In \S\ref{sec:utz} Theorem \ref{teo:exp+} we extend the result known as \emph{Utz's Theorem}, proved in \cites{CK,RW}, that does not exist positively expansive homeomorphisms in compact infinite metric spaces. Finally, in \S\ref{sec:shift} we briefly discuss a notion of entropy for lattice morphisms calculating the entropy of non-Hausdorff shifts in Proposition \ref{prop:nHshift-entropy}.

\subsection*{Acknowledgements} We want thank to Alfonso Artigue, Ignacio Monteverde and José Vieitez for their encourage and help to develop our ideas towards the study of expansive systems.

\section{Expansiveness}\label{sec:exp}

In this section we introduce the basic terminology and notation for lattices. We define expansivity for a lattice automorphism in Definition \ref{def:expnsive} and show some basic properties in Proposition \ref{prop:exp_conjugado} and Proposition \ref{prop:exp_potencia}.

\begin{df} A \emph{lattice} is an \emph{partially ordered set} $(\l,\leq)$, that is, $\leq$ a transitive, antisymmetric and reflexive relation on the set $\l$, such that for every pair of elements $u,v\in\l$ there exist (unique) elements $u\sqcup v,u\sqcap v\in\l$ which are the \emph{least upper bound} and the \emph{greatest lower bound}, respectively, of the set $\{u,v\}$. A lattice is called \emph{distributive} if one of the following equivalent conditions hold: 
\begin{enumerate}
 \item $u\sqcap(v\sqcup w)=(u\sqcap w)\sqcup(u\sqcap w)$ for all $u,v,w\in\l$, or\item $u\sqcup(v\sqcap w)=(u\sqcup w)\sqcap(u\sqcup w)$ for all $u,v,w\in\l$.  
\end{enumerate}
A lattice is said to be \emph{bounded} if $\l$ has maximum and minimum elements denoted $1$ and $0$ respectively. In this article the term \emph{lattice} stands for a bounded distributive lattice such that $0\neq1$, unless otherwise is stated.
\end{df}

If $\c\subseteq\l$ is a finite subset of a lattice then it is easy to see that there exist members of $\l$ which are the least upper bound and the greatest lower bound of $\c$, denoted $\bigsqcup\c$ and $\bigsqcap\c$, respectively. In particular we have $\bigsqcup\varnothing=0$ and $\bigsqcap\varnothing=1$.

Note that if $(X,\tau)$ is a topological space then the topology $\l_X=\tau$ ordered by inclusion, that is $\leq{=}\subseteq$, is a lattice, with minimum element $0=\varnothing$ and maximum $1=X$. The reader interested only in topological dynamics can always assume in what follows that we are working on lattices of this type, with the union and intersection operations as the operations $\sqcup$ and $\sqcap$ of the lattice.

\begin{df} Let $\l$ be a lattice and $\u,\v\subseteq\l$. We say than $\u$ is \emph{finer} than $\v$, or that $\u$ \emph{refines} $\v$ iff for all $u\in\u$ there exists $v\in\v$ such that $u\leq v$, and we write $\u\prec\v$. If $u\in\l$ the notation $u\prec\v$ means that $u\leq v$ for some $v\in\v$. We say that $\u$ is \emph{equivalent} to $\v$, denoted $\u\sim\v$, iff $\u\prec\v$ and $\v\prec\u$. We also define $\u\wedge\v=\{u\sqcap v:u\in\u,v\in\v\}$, and for $u\in\u$ we denote $u\wedge\v=\{u\}\wedge\v$. We say that $\u$ is a \emph{cover} iff $\u$ is a finite set such that $\bigsqcup\u=1$, we write $\c(\l)$ for the set of all covers. If $\u$ is a cover any cover $\v$ such that $\v\subseteq\u$ is called a \emph{subcover} of $\u$.
\end{df}

Note that $\prec$ is a \emph{preorder} (a transitive and reflexive relation) and that $\sim$ is an equivalence relation. The operation $\wedge$ is associative and commutative, $\u\wedge\u\sim\u$ if $\u\subseteq\l$, and $\u\wedge\v$ is a greatest lower bound of $\{\u,\v\}$, that is, $\w\prec\u$ and $\w\prec\v$ for some $\u,\v,\w\subseteq\l$ iff $\w\prec\u\wedge\v$. Note also that by the distributive law if $\u,\v$ are covers then $1=1\sqcap1=(\textstyle\bigsqcup\u)\sqcap(\bigsqcup\v)=\bigsqcup\u\wedge\v$, so that $\u\wedge\v$ is also a cover.

\begin{df}
 Let $\lambda\colon\l\to\l'$ be a map between lattices $\l$ and $\l'$. It is called \emph{monotone} iff for $u,v\in\l$ we have that $u\leq v$ implies $\lambda u\leq\lambda v$, a \emph{morphism} iff $\lambda(u\sqcup v)=\lambda u\sqcup\lambda v$ and $\lambda(u\sqcap v)=\lambda u\sqcap\lambda v$ for all $u,v\in\l$, an \emph{isomorphism} iff it is a bijective morphism, and \emph{unital} if $\lambda(0)=0$ and $\lambda(1)=1$. An \emph{automorphism} of $\l$ is an isomorphism from $\l$ to itself.  If $\u\subseteq\l$ we denote $\lambda\u=\{\lambda u:u\in\u\}$. If $\lambda$ is an automorphism and $n\in\Z$ we define  $\lambda^n$ as the $n$-fold composition $\lambda\circ\cdots\circ\lambda$ if $n>0$, $\lambda^0=\id_\l$ the identity map, and $\lambda^n=(\lambda^{-1})^{-n}$ if $n<0$.
\end{df} 

It is easily checked that composition of morphisms is a morphism and that the inverse map of an isomorphism is an isomorphism, then we see that lattices and their morphisms form a \emph{category} whose isomorphisms are precisely the lattice isomorphisms. It is straightforward to verify that any morphism is monotone, and by \cite{SS81}*{Theorem 2.3} we have that a monotone bijection whose inverse map is also monotone is necessarily an isomorphism.

Let $\lambda\colon\l\to\l'$ a map between lattices and $\u,\v\subseteq\l$. If $\u\prec\v$ and $\lambda$ is monotone then $\lambda\u\prec\lambda\v$. If $\lambda$ is a morphism then $\lambda(\u\wedge\v)=\lambda\u\wedge\lambda\v$. If $\u$ is a cover, $\lambda$ is a morphism and $\lambda(1)=1$ (in particular if $\lambda$ is an isomorphism) then $\lambda\u$ is also a cover.

For topological spaces $(X,\tau_X)$, $(Y,\tau_Y)$ any continuous map $f\colon X\to Y$ induces a map $\lambda_f\colon\l_Y\to\l_X$ between the lattices $\l_Y=\tau_Y$ and $\l_X=\tau_X$ given by $\lambda_f u=f^{-1}u$ if $u$ is an open set in $Y$. We see that $\lambda_f$ is a unital morphism, and $\lambda$ is an isomorphism iff $f$ is a homeomorphism. The assignment $X\mapsto\l_X$, $f\mapsto\lambda_f$ gives a \emph{contra-variant functor} from the category of topological spaces and continuous functions to the category of lattices and their unital morphisms.

\begin{df}\label{def:expnsive}
 Let $\l$ be a lattice and $\lambda\colon\l\to\l$ an automorphism. We say that $\lambda$ is \emph{expansive} iff there exists a cover $\u\in\c(\l)$ with the following property:
 \begin{center}
 for every $\v\in\c(\l)$ there exists $N\in\N$ such that $\bigwedge_{|n|\leq N}\lambda^n\u\prec\v$.
 \end{center}
 Any such cover $\u$ is called \emph{expansivity cover} for $\lambda$.
\end{df}

If $f\colon X\to X$ is a homeomorphism of a compact topological space, and $\lambda_f\colon\l_X\to\l_X$ is the induced lattice automorphism, then $\lambda_f$ is expansive according to Definition \ref{def:expnsive} iff $f$ is \emph{uniformly refinement expansive} as in \cite{AAM}*{Definition 3.6}, where an expansivity cover for $\lambda_f$ is called \emph{uniform refinement expansivity cover} for $f$. By  \cite{AAM}*{Theorems 3.9, 3.13 and 2.7} we have that if $X$ is a compact Hausdorff space then $\lambda_f$ is expansive iff $X$ is metrizable and $f$ is expansive in the usual metric sense.

Note that if $\u$ is an expansivity cover for a lattice automorphism $\lambda$ then any cover finer than $\u$, in particular any cubcover of $\u$, is also an expansivity cover for $\lambda$. This is because $\v\prec\u$ implies $\bigwedge_{|n|\leq N}\lambda^n\v\prec\bigwedge_{|n|\leq N}\lambda^n\u$ if $N\in\N$.

\begin{rmk} In \cite{AH} a notion of expansivity for automorphisms of unital abelian rings is introduced, which is similar to Definition \ref{def:expnsive}. The ordered set $\l$ considered in that article is the collection of ideals ordered by inclusion and the automorphism of $\l$ the one induced by the ring automorphism mapping ideas to ideals. However in the set of ideals the authors considered as the operations $\sqcup$ and $\sqcap$ the sum and the product of ideals. Although distributivity holds for these operations clearly the product of ideals does not give in general a greatest lower bound. 
\end{rmk}

To end this section we present two basic results generalizing well know properties of expansive homeomorphisms.

\begin{prop}\label{prop:exp_conjugado}
 Let $\lambda,\lambda'$ be lattice automorphisms such that $\varphi\lambda=\lambda'\varphi$ for some lattice isomorphism $\varphi$. Then $\lambda$ is expansive iff $\lambda'$ is expansive.
\end{prop}

\begin{proof}
 If $\u$ is an expansivity cover for $\lambda$ then it is easily checked that $\varphi\u$ is an expansivity cover for $\lambda'$. 
\end{proof}

\begin{prop}\label{prop:exp_potencia}
 Let $\lambda$ be a lattice automorphism and $m\in\Z$, $m\neq0$. Then $\lambda$ is expansive iff $\lambda^m$ is expansive.
\end{prop}

\begin{proof}
 As by Definition \ref{def:expnsive} we clearly have that $\lambda$ is expansive iff $\lambda^{-1}$ is expansive, we can assume $m>0$. If $\lambda$ is expansive let $\u$ be an expansivity cover for $\lambda$ and define the cover $\u_m=\tbigwedge_{k=0}^{m-1}\lambda^k\u$. For every $N\in\N$ we have
 \begin{equation*}
 \begin{split}
  \tbigwedge_{|n|\leq N}\lambda^{mn}\u_m
  &=\tbigwedge_{|n|\leq N}\lambda^{mn}\bigl(\tbigwedge_{k=0}^{m-1}\lambda^k\u\bigr)
  =\tbigwedge_{|n|\leq N}\bigl(\tbigwedge_{k=0}^{m-1}\lambda^{mn+k}\u\bigr)\\
  &=\tbigwedge_{n=-mN}^{mN+m-1}\lambda^n\u
   \prec\tbigwedge_{|n|\leq mN}\lambda^n\u. 
 \end{split}
 \end{equation*}
 Then we see that $\u_m$ is an expansivity cover for $\lambda^m$ which is therefore expansive. Conversely, if $\lambda^m$ is expansive then $\lambda$ is expansive as well, because if $\u\in\c(\l)$ and $N\in\N$ clearly we have $\tbigwedge_{|n|\leq mN}\lambda^n\u\prec\tbigwedge_{|n|\leq N}\lambda^{mn}\u$. Then any expansivity cover $\u$ for $\lambda^m$ is also an expansivity cover for $\lambda$.
\end{proof}

\section{Mañé's Theorem}\label{sec:mañé}

In \cite{Ma} Mañé proves that a compact metric space admitting an expansive homeomorphism has finite topological dimension. In this section we consider the generalization of this result to the context of lattice automorphisms in Theorem \ref{teo:mañé}.

\esp The concept of dimension we use here is the obvious natural version for lattices of the \emph{Lebesgue's covering dimension} as defined for example in \cite{Na}*{Definition I.4}.

\begin{df}\label{def:dimensión}
Let $\l$ be a lattice and $\u\in\c(\l)$ a cover. The \emph{order} of $\u$ is
$$
\ord\u=\max\{n\in\N:\exists\;\v\subseteq\u\text{ such that }|\v|=n\text{ and }\tbigsqcap\v\neq0\},
$$
where $|\cdot|$ denotes cardinality. The \emph{dimension} of $\l$ is defined as
$$
\dim\l=\min\{n\in\N:\forall\;\u\in\c(\l)\;\exists\;\v\in\c(\l)\mid\v\prec\u,\,\ord\v\leq n+1\},
$$
if the latter set is not empty, and $\dim\l=\infty$ if it is.
\end{df}

Note that if $X$ is a compact topological space and $\l_X$ is the associated lattice then $\dim\l_X$ as defined before coincides with the covering dimension of $X$. 

\esp\par The following result corresponds to \cite{Ma}*{\S2, Lemma I} of Mañé's paper.

\begin{lem}\label{lem:mañé}
Let $\lambda\colon\l\to\l$ be an expansive lattice automorphism and $\u$ an expansivity cover for  $\lambda$. Then, for every cover $\v$ there exists a cover $\w$ such that
$$
\lambda^{-n}\w\wedge\bigl(\tbigwedge_{|k|\leq n}\lambda^k\u\bigr)\wedge\lambda^n\w\prec\tbigwedge_{|k|\leq n}\lambda^k\v,
$$
for all $n\in\N$.
\end{lem}

\begin{proof}
Given $\v\in\c(\l)$ Let $N\in\N$ be such that $\tbigwedge_{|k|\leq N}\lambda^k\u\prec\v$ and define $\w=\tbigwedge_{|k|\leq N}\lambda^k\u$. Then, for all $n\in\N$ and $l\in\Z$ such that $|l|\leq n$ we have
\begin{equation*}
\begin{split}
\lambda^{-n}\w\wedge\bigl(\displaystyle\tbigwedge_{|k|\leq n}\lambda^k\u\bigr)&\wedge\lambda^n\w=\\
&=\bigl(\displaystyle\tbigwedge_{k=-N-n}^{N-n}\lambda^k\u\bigr)\wedge\bigl(\displaystyle\tbigwedge_{k=-n}^n\lambda^k\u\bigr)\wedge\bigl(\displaystyle\tbigwedge_{k=n-N}^{n+N}\lambda^k\u\bigr)\\
&\prec\displaystyle\tbigwedge_{k=-n-N}^{n+N}\lambda^k\u
\prec\displaystyle\tbigwedge_{k=l-N}^{l+N}\lambda^k\u\\
&=\lambda^l\bigl(\displaystyle\tbigwedge_{|k|\leq N}\lambda^k\u\bigr)
=\lambda^l\w\prec\lambda^l\v.
\end{split}
\end{equation*}
Since $|l|\leq n$ was arbitrary we conclude that
$$
\lambda^{-n}\w\wedge\bigl(\displaystyle\tbigwedge_{|k|\leq n}\lambda^k\u\bigr)\wedge\lambda^n\w\prec\displaystyle\tbigwedge_{|l|\leq n}\lambda^l\v,
$$
for all $n\in\N$ and we are done.
\end{proof}

Let $\l$ be a lattice and consider a cover $\u\in\c(\l)$. We denote
$$
\u^2=\{u_1\sqcup u_2:u_1,u_2\in\u,u_1\sqcap u_2\neq0\}.
$$
Note that $\u^2\in\c(\l)$ and $\u\prec\u^2$ because if $u\in\u$ and $u\neq0$ then $u\in\u^2$. If $\v\in\c(\l)$ then $(\u\wedge\v)^2\prec\u^2\wedge\v^2$: indeed, if $w\in(\u\wedge\v)^2$ then $w=(u_1\sqcap v_1)\sqcup(u_2\sqcap v_2)$ for some $u_1,u_2\in\u$, $v_1,v_2\in\v$ such that $u_1\sqcap v_1\sqcap u_2\sqcap v_2\neq0$, hence $u_1\sqcap u_2\neq0$ and $v_1\sqcap v_2\neq0$, and therefore $w'=(u_1\sqcup u_2)\sqcap(v_1\sqcup v_2)\in\u^2\wedge\v^2$ and as $w\leq w'$ we are done. Note also that if $\lambda$ is an automorphism of $\l$ then $(\lambda\u)^2=\lambda(\u^2)$. Combining these properties we obtain the result that we remark next for future reference.  
 
\begin{rmk}\label{rmk:mañe}
 If $\lambda\colon\l\to\l$ is a lattice automorphism and $\u\in\c(\l)$ then 
 $$
 \bigl(\tbigwedge_{k\in I}\lambda^k\u\bigr)^2\prec\tbigwedge_{k\in I}\lambda^k(\u^2),
 $$
 for any finite set $I\subseteq\Z$.
\end{rmk}

For a finite subset $\v\subseteq\l$ of a lattice $\l$ an element $c\in\l$ is said to be a \emph{component} of $\v$ iff $c=\tbigsqcup_{i=1}^mv_i$ for some $m\geq1$ and $v_1,\ldots,v_m\in\v$, $v_j\sqcap\tbigsqcup_{i=1}^{j-1}v_i\neq0$ if $2\leq j\leq m$, and $c\sqcap v=0$ if $v\in\v\setminus\{v_1,\ldots,v_m\}$. We denote the set of all components of $\v$ as $c(\v)=\{c\in\l:c\text{ is a component of }\v\}$. Clearly we have $\v\prec c(\v)$, $\bigsqcup c(\v)=\bigsqcup\v$, and $c(\v)$ is \emph{pairwise disjoint}, that is, $c_1\sqcap c_2=0$ if $c_1,c_2\in c(\v)$ and $c_1\neq c_2$.

\esp Now we are ready to introduce the main result of this section.

\begin{thm}\label{teo:mañé}
 If a lattice $\l$ admits an expansive automorphism with an expansivity cover of the form $\u^2$ for some $\u\in\c(\l)$ then $\dim\l$ is finite.
\end{thm}

\begin{proof} 
Let $\lambda$ be a lattice automorphism of $\l$ and $\u\in\c(\l)$ such that $\u^2$ is an expansivity cover for $\lambda$. Then, by Lemma \ref{lem:mañé}, there exists $\w\in\c(\l)$ such that
\begin{equation}\label{ec:dimfin1}
 \lambda^{-n}\w\wedge\bigl(\tbigwedge_{|k|\leq n}\lambda^k\u^2\bigr)\wedge \lambda^n\w\prec\tbigwedge_{|k|\leq n}\lambda^k\u,
\end{equation}
for all $n\in\N$. We claim that $\dim\l<|\w|^2$. Given $n\in\N$ consider the covers
$$
\w_n=\lambda^{-n}\w\wedge \lambda^n\w,
\quad\u_n=\tbigwedge_{|k|\leq n}\lambda^k\u,
\quad
\v_n=\lambda^{-n}\w\wedge\u_n\wedge \lambda^n\w=\u_n\wedge\w_n.
$$
For each $w\in\w_n$ consider $\v_n^w\subseteq\l$ as $\v_n^w=c(w\wedge\v_n)$ the set of components of $w\wedge\v_n=\{w\sqcap v:v\in\v_n\}$, and define $\v_n^\w=\bigcup_{w\in\w_n}\v_n^w$. Note that $\bigsqcup\w_n^w=\bigsqcup w\wedge\v_n=w$ for each $w\in\w_n$ because $\bigsqcup\v_n=1$, hence, as $\bigsqcup\w_n=1$, we have $\bigsqcup\v_n^\w=1$, that is, $\v_n^\w$ is a cover.

To prove the claim we will show that for all $n\in\N$ we have ($a$) $\ord\v_n^\w\leq|\w|^2$, and ($b$) $\v_n^\w\prec\u_n$, so that $\v_n^\w$ can be taken finer than any given cover, because, as $\u\prec\u^2$, $\u$ is an expansivity cover for $\lambda$ and $\u_n=\tbigwedge_{|k|\leq n}\lambda^k\u$ by definition. 

To prove ($b$) suppose that $v\in\v_n^\w$. Then $v\in\v_n^w$ for some $w\in\w_n$. Let $v_1,\ldots,v_m\in w\wedge\v_n$ such that the component $v$ of $w\wedge\v_n$ can be written as $v=\bigsqcup_{i=1}^mv_i$, and $v_j\sqcap\tbigsqcup_{i=1}^{j-1}v_i\neq0$ if $2\leq j\leq m$. On one hand, as $v_1,v_2\prec\v_n\prec\u_n$ and $v_1\sqcap v_2\neq0$, we obtain $v_1\sqcup v_2\prec\u_n^2$. Now, by Remark \ref{rmk:mañe} we have
$$
\u_n^2
=\bigl(\displaystyle\tbigwedge_{|k|\leq n}\lambda^k\u\bigr)^2
\prec\displaystyle\tbigwedge_{|k|\leq n}\lambda^k\u^2,
$$
hence $v_1\sqcup v_2\prec\tbigwedge_{|k|\leq n}\lambda^k\u^2$. On the other hand, as $v_1,v_2\leq w\in\w_n$ we deduce $v_1\sqcup v_2\prec\w_n=\lambda^{-n}\w\wedge \lambda^n\w$. From both facts we get $v_1\sqcup v_2\prec\lambda^{-n}\w\wedge\bigl(\tbigwedge_{|k|\leq n}\lambda^k\u^2\bigr)\wedge\lambda^n\w$, and therefore $v_1\sqcup v_2\prec\u_n$ by the relation (\ref{ec:dimfin1}) above. Similarly, by the same reasoning applied this time to $v_1\sqcup v_2$ and $v_3$ instead of $v_1$ and $v_2$, we have $v_1\sqcup v_2\sqcup v_3\prec\u_n$. Repeating this procedure we conclude that $v=\tbigsqcup_{i=1}^mv_i\prec\u_n$, as desired.

Finally, to show ($a$), note that $\w_n=\lambda^{-n}\w\wedge \lambda^n\w$ has at most $|\w|^2$ members. Then, as for each $w\in\w_n$ we have that $\v_n^w$ is pairwise disjoint, we deduce  that at most $|\w|^2$ members of $\v_n^\w=\bigcup_{w\in\w_n}\v_n^w$ can have a meet different from $0$. That is, $\ord\v_n^\w\leq|\w|^2$.
\end{proof}

\begin{rmk}
 For every expansive homeomorphism of a compact metric space with expansivity constant $\delta>0$, any finite open cover $\u$ whose members has diameter less than $\delta/2$ satisfies that the members of $\u^2$ has diameter less than $\delta$, and therefore $\u^2$ is a refinement expansivity cover for the homeomorphism as explained in \cite{AAM}*{p.\,110}. Hence, we see that the lattice automorphism associated to an expansive homeomorphism of a compact metric space always verifies the hypothesis of Theorem \ref{teo:mañé}, and consequently we obtain Mañé's Theorem as a particular case. 
\end{rmk}

\section{Utz's Theorem}\label{sec:utz}

In this section, in Theorem \ref{teo:exp+}, we extend to the context of lattice automorphisms the following result known as \emph{Utz's Theorem}: if a compact metric space supports a positively expansive homeomorphism then it is finite. We refer to \cites{CK,RW} for a proof of this assertion and the definition of positively expansive homeomorphism.

\begin{df}\label{def:exp+}
 A lattice automorphism $\lambda\colon\l\to\l$ is said to be \emph{positively expansive} iff there exists a cover $\u\in\c(\l)$ such that
\begin{center}
 for every $\v\in\c(\l)$ there exists $N\in\N$ such that $\bigwedge_{n=0}^ N\lambda^n\u\prec\v$.
 \end{center}
 Any such cover is called \emph{positive expansivity cover} for $\lambda$.
\end{df}

It can be shown that if $f$ is a homeomorphism of a compact metric space $X$, then $f$ is positively expansive in the usual metric sense iff the induced lattice automorphism $\lambda_f\colon\l_X\to\l_X$ is positively expansive according to Definition \ref{def:exp+}.  

\esp The following result corresponds to \cite{AAM}*{Lemma 3.19}.

\begin{lem}\label{lem:exp+}
If $\lambda\colon\l\to\l$ is a positively expansive lattice automorphism then there is $\u_0\in\c(\l)$ such that if $\v\in\c(\l)$ we have $\lambda^n\u_0\prec\v$ for some $n\in\N$. 
\end{lem}

\begin{proof}
Let $\u$ be a positive expansivity cover for $\lambda$. By Definition \ref{def:exp+} there exists $N\in\N$ such that $\bigwedge_{n=0}^N\lambda^n\u\prec\lambda^{-1}\u$. Applying $\lambda$ to this relation we get $\bigwedge_{n=1}^{N+1}\lambda^n\u\prec\u$, and hence $\bigwedge_{n=1}^{N+1}\lambda^n\u\sim\bigwedge_{n=0}^{N+1}\lambda^n\u$. This proves that we have $\bigwedge_{n=m}^{N+m}\lambda^n\u\sim\bigwedge_{n=0}^{N+m}\lambda^n\u$ for $m=1$. We claim that this is true for every $m\in\N$. Proceeding by induction suppose that $m\geq2$ and that the latter equivalence is true when $m$ is replaced by $m-1$. Then using the case $m=1$ and the induction hypothesis we have
\begin{equation*}
\begin{split}
\tbigwedge_{n=0}^{N+m}\lambda^n\u
&=\bigl(\tbigwedge_{n=0}^{N+1}\lambda^n\u\bigr)\wedge\bigl(\tbigwedge_{n=N+2}^{N+m}\lambda^n\u\bigr)
\sim\bigl(\tbigwedge_{n=1}^{N+1}\lambda^n\u\bigr)\wedge\bigl(\tbigwedge_{n=N+2}^{N+m}\lambda^n\u\bigr)\\
&=\lambda\big(\tbigwedge_{n=0}^{N+m-1}\lambda^n\u\bigr)
\sim\lambda\bigl(\tbigwedge_{n=m-1}^{N+m-1}\lambda^n\u\bigr)
=\tbigwedge_{n=m}^{N+m}\lambda^n\u,
\end{split}
\end{equation*}
proving the claim. Now, taking $\u_0=\bigwedge_{n=0}^N\lambda^n\u$ we can restate what we proved as $\lambda^m\u_0\sim\tbigwedge_{n=0}^{N+m}\lambda^n\u$ for all $m\in\N$. This relation together with the fact that $\u$ is a positive expansivity cover for $\lambda$ implies that $\u_0$ has the desired property.
\end{proof}

An element $u$ of a lattice $\l$ is called a \emph{proper maximal} element iff $u\neq1$ and $u\leq v$ implies $v=u$ or $v=1$ for all $v\in\l$.

\begin{lem}\label{lem:maximales}
 Let $u_0,\ldots,u_m$ different proper maximal elements of a lattice $\l$ and consider $v_0,\ldots,v_m\in\l$ given by $v_k=\tbigsqcap_{i=0,i\neq k}^mu_i$ for $k=0,\ldots,m$. Then $\v=\{v_0,\ldots,v_m\}$ is a cover with $m+1$ members that has no proper subcover. 
\end{lem}

\begin{proof}
 The proof proceed by induction on the cardinality $n=|\{u_0,\ldots,u_m\}|=m+1$. If $n=1$, that is $m=0$, then $v_0=\tbigsqcap\varnothing=1$, and clearly $\v=\{1\}$ has no proper subcover. Now suppose that the property is true for a certain cardinality $n-1=m\geq1$ and consider $u_0,\ldots,u_m$ different proper maximal elements of $\l$. We have
 \begin{equation*}
 \begin{split}
 \tbigsqcup_{k=0}^mv_k
 &\geq\tbigsqcup_{k=1}^mv_k
 =\tbigsqcup_{k=1}^m\tbigsqcap_{i=0,i\neq k}^mu_i
 =\tbigsqcup_{k=1}^mu_0\sqcap\tbigsqcap_{i=1,i\neq k}^mu_i\\
 &=u_0\sqcap\tbigsqcup_{k=1}^m\tbigsqcap_{i=1,i\neq k}^mu_i
 =u_0\sqcap1=u_0,
 \end{split}
 \end{equation*}
 where we applied the induction hypothesis to $u_1,\ldots,u_m$ to get $\tbigsqcup_{k=1}^m\tbigsqcap_{i=1,i\neq k}^mu_i=1$. This shows $\tbigsqcup_{k=0}^mv_k\geq u_0$. Similarly we have $\tbigsqcup_{k=0}^mv_k\geq u_1$, hence, as $u_0,u_1$ are different proper maximal elements, we obtain $\tbigsqcup_{k=0}^mv_k=1$, that is, $\v=\{v_0,\ldots,v_m\}$ is a cover. Finally, in the above calculation we showed that $\tbigsqcup_{k=1}^mv_k=u_0\neq1$, so that $v_0$ cannot be removed from $\v$ to get a cover. Analogously, no $v_k$ with $k=0,\ldots,m$ can be removed from $\v$, hence $\v$ has $m+1$ members and no proper subcover. 
\end{proof}

The following is the principal result of this section, inspired on \cite{AAM}*{Theorem 3.20}.

\begin{thm}\label{teo:exp+}
 Let $\l$ be a lattice that admits a positively expansive automorphism. Then there exists $N\in\N$ such that every cover has a subcover with at most $N$ members. In particular $\l$ has finitely many proper maximal elements. 
\end{thm}

\begin{proof}
Let $\u_0$ be a cover as in Lemma \ref{lem:exp+} and $N=|\u_0|$. As any cover is refined by some iterate $\lambda^n\u_0$, $n\in\N$, we see that all covers have a subcover of cardinality less or equal to $|\lambda^n\u_0|=N$. For the second assertion, note that by Lemma \ref{lem:maximales} there can exist at most $N$ proper maximal elements. 
\end{proof}

Applying Theorem \ref{teo:exp+} to the lattice automorphism induced by a \emph{positively refinement expansive} homeomorphism (see \cite{AAM}*{Definition 3.17}) we obtain \cite{AAM}*{Theorem 3.20}, because in a topological space having the $T_1$ separation property the proper maximal open sets are the complements of the points. In particular, considering the compact metric space case, we see that Theorem \ref{teo:exp+} implies Utz's Theorem.

\section{Entropy}\label{sec:shift}

In this section we discuss the concept of \emph{entropy} for lattice morphisms. We recall the definition of \emph{non-Hausdorff shifts} and calculate their entropy in Proposition \ref{prop:nHshift-entropy}.

\esp The notion of \emph{topological entropy} for a continuous map $f\colon X\to X$ on a compact topological space $X$ is introduced in \cite{AKM}. In \S5 of that article the authors suggest that the definition of entropy can be formulated in an abstract setting, showing an example in the category of abelian groups. A first evident general scenario where we can state a definition of entropy is the context of lattices and their morphisms.

\begin{df}\label{def:entropía}
Let $\l$ be a lattice and $\lambda\colon\l\to\l$ a unital morphism. Given $\u\in\c(\l)$ we define the \emph{entropy of $\u$} as
$$
h(\u)=\log\bigl(\min\bigl\{|\v|:\v\in\c(\l),\v\subseteq\u\bigr\}\bigr),
$$
and the \emph{entropy of $\lambda$ relative to $\u$} as
$$
h(\lambda,\u)=\lim_{n\to\infty}{\textstyle\frac1n}h\bigl(\tbigwedge_{k=0}^{n-1}\lambda^k\u\bigr).
$$
The \emph{entropy of $\lambda$} is defined as
$$
h(\lambda)=\sup\bigl\{h(\lambda,\u):\u\in\c(\l)\bigr\},
$$
where we agree that the $\sup(\cdot)$ function can take values in $[0,+\infty]$.
\end{df}

Clearly much of the theory developed in \cite{AKM} remains valid, word by word, in the context of lattices. In particular, as in \cite{AKM}*{Property 8}, the sequence $a_n=h\bigl(\tbigwedge_{k=0}^{n-1}\lambda^k\u\bigr)$ is sub-additive, so that the limit in Definition \ref{def:entropía} exists. Also the following result form \cite{KR} holds.

\begin{prop}[\cite{KR}*{Theorem 2.6}]\label{prop:exp entropia}
If $\lambda\colon\l\to\l$ is an expansive lattice automorphism with expansivity cover $\u$ then $h(\lambda)=h(\lambda,\u)$. In particular $h(\lambda)$ is finite.
\end{prop}

Next we recall the definition of the \emph{non-Hausdorff shifts} from \cite{AAM}*{\S4.1}.

\esp Let $\s$ a nonempty finite set whose elements we call {\em symbols}, and let $\leq$ be a partial order on $\s$. We say that $(\s,\leq)$ is {\em lower complete} iff every nonempty lower bounded subset of $\s$ has a (unique) greatest lower bound. For convenience we introduce a new symbol, say $0\notin\s$, and agree that $0\leq u$ for every $u\in\s\cup\{0\}$.

Given a lower complete set of symbols $\s$ we consider the \emph{order topology} on $\s$, that is, the topology that has the base $\b=\bigl\{(0,u]:u\in\s\cup\{0\}\bigr\}$, where we denote $(0,u]=\{v\in\s:v\leq u\}$ for $u\in\s\cup\{0\}$. Let $\Sigma$ the product space $\Sigma=\s^{\Z}$ endowed with the product topology and consider the shift homeomorphism given by
$$
\sigma\colon\Sigma\to\Sigma,\qquad\sigma\bigl((u_k)_{k\in\Z}\bigr)=(u_{k+1})_{k\in\Z}\quad \text{for}\quad(u_k)_{k\in\Z}\in\Sigma.
$$
Note that for the special case in which the relation $\leq$ is the identity of $\s$, the order topology of $\s$ becomes the discrete topology and $(\Sigma,\sigma)$ is the usual shift which is a compact metric expansive dynamical system. In this case we denote the shift as $\sigma_{met}$. In the general case $(\Sigma,\sigma)$ is a refinement expansive homeomorphism as showed in \cite{AAM}*{Theorem 4.2}.

\esp Our ext result generalizes \cite{AKM}*{Example 3} where the entropy of the metric shift is calculated obtaining $h(\sigma_{met})=\log|\s|$. 

\begin{prop}\label{prop:nHshift-entropy}
Let $\s$ a lower complete set of symbols and consider the system $(\Sigma,\sigma)$ defined above. Then $h(\sigma)=\log|\s_+|$, where $\s_+$ is the set of maximal elements of $\s$.
\end{prop}

\begin{proof}
By \cite{AAM}*{Theorem 4.2} and its proof $\sigma$ is a refinement expansive homeomorphism and the collection of $0$-cylinders $\u=\{C(u):u\in\s\}$, where $C(u)=\{(v_k)_{k\in\Z}\in\Sigma:v_0\leq u\}$, is an expansivity cover for $\sigma$. Then, by Proposition \ref{prop:exp entropia} we have $h(\sigma)=h(\sigma,\u)$. Given $n\in\N$ we have
$$
\tbigwedge_{k=0}^{n-1}\sigma^{-k}\u=\bigl\{C(u_0,\ldots,u_{n-1}):u_0,\ldots,u_{n-1}\in\s\bigr\},
$$
where $C(u_0,\ldots,u_{n-1})=\bigl\{(v_k)_{k\in\Z}\in\Sigma: v_k\leq u_k\text{ if }0\leq k\leq n-1\bigr\}$ for $u_0,\ldots,u_{n-1}\in\s$. It is easy to check that the cover
$$
\v=\bigl\{C(u_0,\ldots,u_{n-1}):u_0,\ldots,u_{n-1}\in\s_+\bigr\},
$$
is a subcover of $\tbigwedge_{k=0}^{n-1}\sigma^{-k}\u$ of minimal cardinality equal to $|\s_+|^n$. Then
 $$
 h(\sigma,\u)
 =\lim_{n\to\infty}\tfrac1nh\Bigl(\tbigwedge_{k=0}^{n-1}\sigma^{-k}\u\Bigr)
 =\lim_{n\to\infty}\tfrac1n\log\bigl(|\s_+|^n\bigr)
 =\log|\s_+|,
 $$
 as desired. 
\end{proof}

\vspace{20mm}

\noindent Mauricio Achigar,\\
{\tt machigar@unorte.edu.uy},\\
{\sc Departamento de Matemática y Estadística del Litoral},\\
{\sc Centro Universitario Regional Litoral Norte},\\
{\sc Universidad de la República.}\\
25 de Agosto 281, Salto (50000), Uruguay.

\end{document}